\documentclass{endm}
\usepackage{endmmacro}
\usepackage{graphicx}

\def\diam{{\rm diam}}
\def\span{{\rm span}}
\def\rn{{\rm rn}}

\def\vp{\varphi}

\begin{document}

\begin{verbatim}\end{verbatim}\vspace{2.5cm}

\begin{frontmatter}

\title{Radio number for middle graph of paths}

\author{Devsi Bantva\thanksref{myemail}}
\address{Department of Mathematics\\ Lukhdhirji Engineering College, Morvi - 363 642, Gujarat, India}

\thanks[myemail]{Email:\href{mailto:devsi.bantva@gmail.com} {\texttt{\normalshape devsi.bantva@gmail.com} (Devsi Bantva)}}

\begin{abstract}
For a connected graph $G$, let $\diam(G)$ and $d(u,v)$ denote the diameter of $G$ and distance between $u$ and $v$ in $G$. A radio labeling of a graph $G$ is a mapping $\vp: V(G) \rightarrow \{0,1,2,...\}$ such that $|\vp(u)-\vp(v)| \geq \diam(G) + 1 - d(u,v)$ for every pair of distinct vertices $u, v$ of $G$. The span of $\vp$ is defined as span($\vp$) = max\{$|\vp(u)-\vp(v)|$ : $u, v \in V(G)$\}. The radio number $rn(G)$ of $G$ is defined as $\rn(G)$ = min\{span($\vp$) : $\vp$ is a radio labeling of $G$\}. In this paper, we determine the radio number for middle graph of paths. \\

\end{abstract}


\end{frontmatter}

\section{Introduction}\label{intro}

The channel assignment to radio transmitters is one of the main objectives in setup of wireless communication system. A proper channel assignment to radio transmitters which satisfies interference constraints with maximum use of spectrum is a need of wireless communication system. The interference constraints between a pair of transmitters is closely related with separation of channels and distance between transmitters. In a network, if two transmitters are closer then higher the interference between them and large separation between assign channels to these transmitters is required to mitigate the interference but if the distance between two transmitters is large enough then the same channel or channels with small separation can be use for both transmitters.

In a graph model, introduced by Hale\cite{Hale} in 1980 and later modified by Roberts\cite{Roberts}, the transmitters are represented by vertices of a graph and the interference between a pair of vertices by an edge, such graph model is known as interference graph. The assignment of channels is converted into graph labeling (a graph labeling is an assignment of label to each vertex according to certain rule). Initially only two level interference, namely major and minor, was considered and accordingly two transmitters are classified as very close and close in network. In an interference graph two vertices are very close if they are adjacent and close if they are at distance two apart.

Roberts\cite{Roberts} proposed that a pair of transmitters which has minor interference must receive different channels and a pair of transmitters which has major interference must receive channels that are at least two apart. Motivated through this problem Griggs and Yeh\cite{Griggs} introduced distance two labeling as follows: A distance two labeling (or $L(2,1)$-labeling) of a graph $G$ = $(V(G), E(G))$ is a function $\vp$ from vertex set $V(G)$ to the set of non-negative integers such that $|\vp(u)-\vp(v)| \geq 2$ if $d(u,v)$ = 1 and $|\vp(u)-\vp(v)| \geq 1$ if $d(u,v)$ = 2. The span of $\vp$ is defined as max\{$|\vp(u)-\vp(v)|$ : $u, v \in V(G)$\}. The $\lambda$-number for a graph $G$, denoted by $\lambda(G)$, is the minimum span of a distance two labelings for $G$. The $L(2,1)$-labeling is explored in past two decades by many researchers (see survey articles \cite{Calamoneri} and \cite{Yeh1} for detail).

But as time passed, Chartrand et al.\cite{Chartrand1} introduced the concept of radio labeling by extending the level of interference from two to the largest possible-the diameter of graph $G$ which generalize distance two labeling of graphs; defined as follows.

\begin{definition}
A \emph{radio labeling} of a graph $G$ is a mapping $\vp: V(G) \rightarrow \{0, 1, 2, \ldots\}$ such that for every pair of distinct vertices $u, v$ of $G$,
$$
d(u,v) + |\vp(u)-\vp(v)| \geq \diam(G) + 1.
$$
The \emph{radio number} of $G$ is defined as
$$
\rn(G) := \min_{\vp} \span(\vp)
$$
with minimum taken over all radio labelings $\vp$ of $G$. A radio labeling $\vp$ of $G$ is \emph{optimal} if $\span(\vp) = \rn(G)$. If $\diam(G)$ = 2 then $\rn(G)$ = $\lambda(G)$.
\end{definition}

Determining the radio number of a graph is an interesting but challenging problem. So far the radio number is known only for a handful families of graphs; see \cite{Chartrand} for a survey. Recently, author with others also studied the radio number for trees in \cite{Bantva1,Bantva2}. In this paper, by using the approach similar to one used in \cite{Vaidya1} by Vaidya and Bantva, we determine the radio number for middle graph of paths. However, it should be notice that in \cite{Vaidya1}, authors determined the exact radio number for total graph of paths $P_{2k}$ and gave bounds for the radio number for total graph of paths $P_{2k+1}$ while in this paper, the radio number is completely determined for middle graph of paths $P_{n}$ for any $n$. We follow \cite{West} for graph theoritic terms and notation.

\section{Main Results}

The radio number of paths is investigated by Liu and Zhu in \cite{Liu} as stated in the following result. Note that investigating the radio number for paths which is a very basic graph family, was challenging.
\begin{theorem}\cite{Liu} For any $n \geq 3$,
\begin{eqnarray*}
rn(P_{n}) & = &\left\{
\begin{array}{l}
\begin{tabular}{lll}
$2k(k-1)+1$, & if $n = 2k$, \\
$2k^{2}+2$, & if $n = 2k+1$.
\end{tabular}
\end{array}
\right.
\end{eqnarray*}
\end{theorem}

The middle graph of a graph $G$, denoted by $M(G)$, is the graph whose vertex set is $V(G) \cup E(G)$ and in which two vertices are adjacent if and only if either they are adjacent edges of $G$ or one is a vertex of $G$ and the other is an edge incident on it. Let $P_{n}$ be a path on $n$ vertices and $M(P_{n})$ denotes the middle graph of path $P_{n}$ then $|V(M(P_{n}))| = 2n-1$. In this work, denote $p = |V(M(P_{n}))|$. The $\lambda$-number for middle graph of paths is determined by Vaidya and Bantva in \cite{Vaidya2} as stated in the following result.

\begin{theorem}\cite{Vaidya2} For the middle graph $M(P_{n})$ of path $P_{n}$,
\begin{eqnarray*}
\lambda(M(P_{n})) & = &\left\{
\begin{array}{l}
\begin{tabular}{lll}
3, & if $n = 2$, \\
4, & if $n = 3$, \\
5, & if $n = 4,5$, \\
6, & if $n \geq 6$. \\
\end{tabular}
\end{array}
\right.
\end{eqnarray*}
\end{theorem}

Now we focus on to derive the radio number for middle graph of paths. Let $v_{1}$, $v_{2}$,...,$v_{n}$ are vertices and $e_{i}$ = $v_{i}v_{i+1}$, $1 \leq i \leq n-1$ are edges of path $P_{n}$ then \{$v_{1}$, $v_{2}$,...,$v_{n}$, $e_{1}$, $e_{2}$,...,$e_{n-1}$\} is the vertex set of $M(P_{n})$. We rename the vertices $e_{i}$ by $v_{i}^{'}$, $1 \leq i \leq n-1$ in $M(P_{n})$. Note that the diameter of $M(P_{n})$ is $n$ and the center of graph $M(P_{n})$ is $K_{1}$ when $n$ is even and $K_{3}$ when $n$ is odd. A vertex $v$ is called a central vertex of $M(P_{n})$ if $v \in V(C(M(P_{n})))$. Note that $v_{n/2}^{'}$ is the central vertex of $M(P_{n})$ when $n$ is even and $v_{\lfloor n/2 \rfloor}^{'}$, $v_{\lfloor n/2 \rfloor+1}^{'}$ and $v_{\lfloor n/2 \rfloor+1}$ are central vertices when $n$ is odd.

Define the level function $L$ : $V(M(P_{n}))$ $\rightarrow$ \{0,1,2,...\} by

$L(u)$ = min\{$d(w,u)$ : $w$ $\in$ $V(C(M(P_{n})))$\} for any $u$ $\in$ $V(M(P_{n}))$.

In a graph $M(P_{n})$, the maximum level is $\lfloor n/2 \rfloor$. Moreover, in middle graph $M(P_{n})$ of paths $P_{n}$, observe that
\begin{eqnarray}\label{obs}
d(u,v)& \leq &\left\{
\begin{array}{l}
\begin{tabular}{lll}
$L(u) + L(v)$, & if $n = 2k$, \\
$L(u) + L(v) + 1$, & if $n = 2k+1$.
\end{tabular}
\end{array}
\right.
\end{eqnarray}

Note that a radio labeling $\vp$ of a graph is a one-to-one integral function from the set of vertices to the set of non-negative integers. Hence any radio labeling induces an ordering $u_{1}$, $u_{2}$,...,$u_{p}$ of vertices such that $0 = \vp(u_{1}) < \vp(u_{2}) < ... < \vp(u_{p}) = \span(\vp)$. We first prove the following result which is useful to prove our main Theorem.

\begin{lemma}\label{lemma1}~~Let $\vp$ be an assignment of distinct non-negative  integers to $V(M(P_{n}))$ and $u_{1}$, $u_{2}$, $u_{3}$ ,..., $u_{p}$ be the ordering of $V(M(P_{n}))$ such that $\vp(u_{i})$ $<$ $\vp(u_{i+1})$ defined by $\vp(u_{1})$ = 0 and $\vp(u_{i+1})$ = $\vp(u_{i}) + d + 1 - d(u_{i}, u_{i+1})$. If $d(u_{i}, u_{i+1})$ $\leq$ $k+1$ for any $1$ $\leq$ $i$ $\leq$ $p-1$, where $k$ = $\lfloor n/2 \rfloor$ then $\vp$ is a radio labeling.
\end{lemma}
\begin{proof} Let $\vp$ be an assignment of distinct non-negative integers to $V(M(P_{n}))$ such that $\vp(u_{1})$ = 0, $\vp(u_{i+1})$ = $\vp(u_{i}) + d + 1 - d(u_{i}, u_{i+1})$ and $d(u_{i}, u_{i+1})$ $\leq$ $k+1$ for any $1$ $\leq$ $i$ $\leq$ $p-1$ holds, where $k$ = $\lfloor n/2 \rfloor$.

We want to prove that $\vp$ is a radio labeling. $i.e.$ for any $i \neq j$, $|\vp(u_{j}) - \vp(u_{i})|$ $\geq$ $d + 1 - d(u_{i}, u_{j})$.

For each $i$ = 1, 2, ..., $p-2$, without loss of generality, let $j \geq i+2$ then
\begin{eqnarray*}
\vp(u_{j}) - \vp(u_{i}) & = & (j-i)(d+1) - d(u_{i}, u_{i+1}) - ... - d(u_{j-1}, u_{j}) \\
 & \geq & (j-i)(d+1) - (j-i)(k+1).
\end{eqnarray*}

\textsf{Case - 1 :} $n = 2k$.

In this case $d$ = $2k$ and hence we obtain,
\begin{eqnarray*}
\vp(u_{j}) - \vp(u_{i}) & \geq & (j-i)(2k+1) - (j-i)(k+1) \\
& = & (j-i)(2k+1-k-1) \\
& = & (j-i)k \\
& \geq & 2k \\
& \geq & d + 1 - d(u_{i}, u_{j}) \mbox{ as } d(u_{i}, u_{j}) \geq 1.
\end{eqnarray*}

\textsf{Case - 2 :} $n = 2k+1$.

In this case $d$ = $2k+1$ and hence we obtain,
\begin{eqnarray*}
\vp(u_{j}) - \vp(u_{i}) & \geq & (j-i)(2k+2) - (j-i)(k+1) \\
& = & (j-i)(2k+2-k-1) \\
& = & (j-i)(k+1) \\
&\geq& 2(k+1) \\
&\geq& d + 1 - d(u_{i}, u_{j}) \mbox{ as } d(u_{i}, u_{j}) \geq 1.
\end{eqnarray*}
Thus, in both the cases $\vp$ is a radio labeling and hence the result.
\end{proof}

\begin{theorem}\label{Thm:lower}~~Let $M(P_{n})$ be a middle graph of path $P_{n}$ on $n$ vertices. Then
\begin{eqnarray*}
rn(M(P_{n})) & \geq &\left\{
\begin{array}{l}
\begin{tabular}{lll}
$4k^{2}-1$, & if $n = 2k$, \\
$4k(k+1)$, & if $n = 2k+1$.
\end{tabular}
\end{array}
\right.
\end{eqnarray*}
\end{theorem}
\begin{proof} Let $\vp$ be an optimal radio labeling for $M(P_{n})$ with a linear order $u_{1},u_{2},...,u_{p}$ of vertices of $M(P_{n})$ such that 0 = $\vp(u_{1})$ $<$ $\vp(u_{2})$ $<$ $\vp(u_{3})$ $<$ ... $<$ $\vp(u_{p})$. Then $\vp(u_{i+1}) - \vp(u_{i})$ $\geq$ $(d+1) - d(u_{i}, u_{i+1})$ for all $1 \leq i \leq p-1$. Summing up these $p-1$ inequality we get
\begin{equation}\label{eqn1}
rn(M(P_{n})) = \vp(u_{p})\geq (p-1)(d+1) - \displaystyle \sum_{i = 1}^{p-1} d(u_{i}, u_{i+1})
\end{equation}

\textsf{Case - 1 :} $n$ = $2k$.

For $M(P_{2k})$, using (\ref{obs}), we have
\begin{eqnarray}
\displaystyle \sum_{i = 1}^{p-1} d(u_{i}, u_{i+1}) & \leq & \displaystyle \sum_{i = 1}^{p-1} [L(u_{i}) + L(u_{i+1})] \nonumber\\
& = & 2\displaystyle \sum_{u \in V(G)} L(u) - L(u_{1}) - L(u_{p}) \nonumber\\
& \leq & 2\displaystyle \sum_{u \in V(G)} L(u) - 1 \label{eqn2}
\end{eqnarray}

Substituting (\ref{eqn2}) in (\ref{eqn1}), we get

$rn(M(P_{2k}))$ = $\vp(u_{p})$ $\geq$ $(p-1)(d+1) - 2\displaystyle \sum_{u \in V(G)} L(u) + 1$

For $M(P_{2k})$, $p$ = $4k-1$, $d$ = $2k$ and $\sum_{u \in V(G)} L(u)$ = $2k^{2}$. Substituting these into above equation, we get
\begin{eqnarray*}
rn(M(P_{2k})) &=& \vp(u_{p}) \\
&\geq& (4k-1-1)(2k+1) - 4k^{2} + 1 \\
&=& 4k^{2} - 1
\end{eqnarray*}

\textsf{Case - 2 :} $n$ = $2k+1$.

For $M(P_{2k+1})$, using (\ref{obs}), we have
\begin{eqnarray}\label{eqn3}
\displaystyle \sum_{i = 1}^{p-1} d(u_{i}, u_{i+1}) & \leq &\displaystyle \sum_{i = 1}^{p-1} [L(u_{i}) + L(u_{i+1}) + 1] \nonumber \\
& = & 2\displaystyle \sum_{u \in V(G)} L(u) - L(u_{1}) - L(u_{p}) + (p-1) \nonumber \\
& \leq & 2\displaystyle \sum_{u \in V(G)} L(u) + (p-1)
\end{eqnarray}

Substituting (\ref{eqn3}) in (\ref{eqn1}), we get
\begin{eqnarray*}
rn(M(P_{2k+1})) & = & \vp(u_{p}) \geq (p-1)(d+1) - 2\displaystyle \sum_{u \in V(G)} L(u) - (p-1) \\
& = & (p-1)d - 2\displaystyle \sum_{u \in V(G)} L(u)
\end{eqnarray*}
For $M(P_{2k+1})$, $p$ = $4k+1$, $d$ = $2k+1$ and $\sum_{u \in V(G)} L(u)$ = $2k^{2}$. Substituting these into above equation, we get
\begin{eqnarray*}
rn(M(P_{2k+1})) &=& \vp(u_{p}) \\
&\geq& (4k+1-1)(2k+1) - 4k^{2} \\
&=& 4k(k+1)
\end{eqnarray*}
Thus, we have
\begin{eqnarray*}
rn(M(P_{n})) & \geq &\left\{
\begin{array}{l}
\begin{tabular}{ll}
$4k^{2}-1$, & if $n = 2k$, \\
$4k(k+1)$, & if $n = 2k+1$.
\end{tabular}
\end{array}
\right.
\end{eqnarray*}
\end{proof}

\begin{theorem}\label{Thm:up}~~Let $M(P_{n})$ be a middle graph of path $P_{n}$ on $n$ vertices. Then
\begin{eqnarray}\label{eq:up}
rn(M(P_{n})) & \leq &\left\{
\begin{array}{l}
\begin{tabular}{ll}
$4k^{2}-1$, & if $n = 2k$, \\
$4k(k+1)$, & if $n = 2k+1$.
\end{tabular}
\end{array}
\right.
\end{eqnarray}
\end{theorem}
\begin{proof} We consider the following two cases.

\textsf{Case - 1 :} $n$ = $2k$.~~For $M(P_{2k})$, define $\vp$ : $V(M(P_{2k}))$ $\rightarrow$ \{0,1,2, ... ,$4k^{2}-1$\} by $\vp(u_{1})$ = 0 and $\vp(u_{i+1})$ = $\vp(u_{i}) + d + 1 - L(u_{i}) - L(u_{i+1})$ for all $1 \leq i \leq p-1$ as per following ordering of vertices.

We first set $u_{1}$ = $v_{k}^{'}$, $u_{p}$ = $v_{k}$ and for $2 \leq i \leq p-1$, set
$u_{j}$ := $v_{i}$, where
\begin{eqnarray*}
j & = &\left\{
\begin{array}{l}
\begin{tabular}{ll}
$4(k-i)+1$, & if $i < k$, \\
$4(2k-i)+2$, & if $i > k$.
\end{tabular}
\end{array}
\right.
\end{eqnarray*}
and $u_{j}$ := $v_{i}^{'}$, where
\begin{eqnarray*}
j & = &\left\{
\begin{array}{l}
\begin{tabular}{ll}
$4(k-i-1)+3$, & if $i < k$, \\
$4(2k-i)$, & if $i > k$.
\end{tabular}
\end{array}
\right.
\end{eqnarray*}

\textsf{Case - 2 :} $n$ = $2k+1$.~~For $M(P_{2k+1})$, define $\vp$ : $V(M(P_{2k+1}))$ $\rightarrow$ \{0,1,2, ... ,$4k(k+1)$\} by $\vp(u_{1})$ = 0 and $\vp(u_{i+1})$ = $\vp(u_{i}) + d - L(u_{i}) - L(u_{i+1})$ for all $1 \leq i \leq p-1$ as per following ordering of vertices.

We first set $u_{1}$ = $v_{k}^{'}$, $u_{p}$ = $v_{k+1}$, $u_{p-1}$ = $v_{n}$, $u_{p-2}$ = $v_{1}$, $u_{p-3}$ = $v_{k+1}^{'}$ and for $2 \leq i \leq p-4$, set
$u_{j}$ := $v_{i}$, where
\begin{eqnarray*}
j & = &\left\{
\begin{array}{l}
\begin{tabular}{ll}
$4(k+1-i)+1$, & if $1 < i \leq k$, \\
$4(2k-i)+2$, & if $k+1 < i < 2k+1$.
\end{tabular}
\end{array}
\right.
\end{eqnarray*}
and $u_{j}$ := $v_{i}^{'}$, where
\begin{eqnarray*}
j & = &\left\{
\begin{array}{l}
\begin{tabular}{ll}
$4(k-i-1)+3$, & if $i < k$, \\
$4(2k+1-i)$, & if $i > k+1$.
\end{tabular}
\end{array}
\right.
\end{eqnarray*}

Thus, in each case above, we have defined a linear order $u_{1}, u_{2},...,u_{p}$ of the vertices of $M(P_{n})$ which satisfies the conditions of Lemma \ref{lemma1} (note that in case of $M(P_{2k+1})$, $d(u_{p-2},u_{p-1}) \not\leq k+1$ but it is easy to verify radio labeling condition for these vertices) whose span equal to the right-hand side of (\ref{eq:up}). Thus, we have
\begin{eqnarray*}
rn(M(P_{n})) & \leq &\left\{
\begin{array}{l}
\begin{tabular}{lll}
$4k^{2}-1$, & if $n = 2k$, \\
$4k(k+1)$, & if $n = 2k+1$.
\end{tabular}
\end{array}
\right.
\end{eqnarray*}
\end{proof}

\begin{theorem} Let $M(P_{n})$ be a middle graph of path $P_{n}$ on $n$ vertices. Then
\begin{eqnarray*}
rn(M(P_{n})) & = &\left\{
\begin{array}{l}
\begin{tabular}{ll}
$4k^{2}-1$, & if $n = 2k$, \\
$4k(k+1)$, & if $n = 2k+1$.
\end{tabular}
\end{array}
\right.
\end{eqnarray*}
\end{theorem}
\begin{proof} The proof follows from Theorem \ref{Thm:lower} and Theorem \ref{Thm:up}.
\end{proof}

\end{document}